\documentclass[11pt]{amsart}
\usepackage{amssymb,enumerate}
\usepackage{amsmath}
\usepackage{bbm}           
\usepackage{bm}
\usepackage{eso-pic,graphicx}
\usepackage{tikz}
\usepackage{comment}

\usepackage[colorlinks=true, pdfstartview=FitV, linkcolor=blue, citecolor=blue, urlcolor=blue]{hyperref}

\makeatletter
\def\Ddots{\mathinner{\mkern1mu\raise\p@
\vbox{\kern7\p@\hbox{.}}\mkern2mu
\raise4\p@\hbox{.}\mkern2mu\raise7\p@\hbox{.}\mkern1mu}}
\makeatother

\def\Xint#1{\mathchoice
{\XXint\displaystyle\textstyle{#1}}%
{\XXint\textstyle\scriptstyle{#1}}%
{\XXint\scriptstyle\scriptscriptstyle{#1}}%
{\XXint\scriptscriptstyle\scriptscriptstyle{#1}}%
\!\int}
\def\XXint#1#2#3{{\setbox0=\hbox{$#1{#2#3}{\int}$}
\vcenter{\hbox{$#2#3$}}\kern-.5\wd0}}

\def\dashint{\Xint-}

\newtheorem{theorem}{Theorem}[section]

\theoremstyle{definition}

\newtheorem{remark}[theorem]{Remark}

\def\R{\mathbb R}

\def\bey{\begin{eqnarray*}}
\def\eey{\end{eqnarray*}}

\def\rb{\mathcal R}

\newcommand{\eps}{\varepsilon}
\def\({\left(}
\def\){\right)}
\def\[{\left[}
\def\]{\right]}
\def\<{\langle}
\def\>{\rangle}

\newcommand{\bmo}{{\rm BMO}}

\newcommand{\cmo}{{\rm CMO}}

\begin{document}



\title[Characterization of compactness of bilinear commutators]{Characterization
of compactness of commutators of \\
bilinear singular integral
operators}

\date{\today}

\author[L. Chaffee]{Lucas Chaffee}
\address{%
Department of Mathematics \\
Western Washington University\\
516 High Street\\
Bellingham, WA 98225, USA} \email{Lucas.Chaffee@wwu.edu}

\author[P. Chen]{Peng Chen}
\address{ Department of Mathematics, Sun Yat-sen (Zhongshan)
University, Guang- zhou, 510275, P.R. China}
\email{chenpeng3@mail.sysu.edu.cn}

\author[Y. Han]{Yanchang Han}
\address{School of Mathematical Sciences, South China Normal University,
Guang- zhou, 510631, P.R. China}
\email{hanych@scnu.edu.cn}

\author[R.H. Torres]{Rodolfo H. Torres}
\address{%
Department of Mathematics\\
University of Kansas\\
Lawrence, KS 66045, USA}
\email{torres@ku.edu}

\author[L.A. Ward]{Lesley A. Ward}
\address{School of Information Technology and Mathematical Sciences,
         University of South Australia,
         Mawson Lakes  SA  5095,
         Australia
         }
\email{lesley.ward@unisa.edu.au}

\thanks{Peng Chen was supported by NNSF of China 11501583,
Guangdong Natural Science Foundation 2016A030313351, the
Fundamental Research Funds for the Central Universities
161gpy45 and by the Australian Research Council, Grant
No.~ARC-DP160100153. Yanchang Han was supported by Guangdong
Province Natural Science Foundation  Grant No.~2017A030313028
and Lesley A. Ward was supported by the Australian Research
Council, Grant No.~ARC-DP160100153.}

\begin{abstract}
    The commutators of  bilinear Calder\'on-Zygmund operators and
    point-wise multiplication with a symbol in $\cmo$ are bilinear
    compact operators on product of Lebesgue spaces. This  work
    shows that, for certain non-degenerate Calder\'on-Zygmund
    operators, the symbol being in $\cmo$ is not only sufficient
    but actually necessary for the compactness of the commutators.
\end{abstract}

\maketitle

\section{Introduction}

In this note we resolve a problem that has been open for a
while in  the multilinear Calder\'on--Zygmund theory. Namely,
whether the compactness of the  commutators of the bilinear
Riesz transforms (see the next section for technical
definitions) with point-wise multiplication can be used to
characterize the space $\cmo(\R^n)$. For the purpose of this
article,  $\cmo(\R^n)$ is the closure in the John--Nirenberg
$\bmo(\R^n)$, with its usual topology, of the space of
infinitely differentiable functions with compact support. This
problem has been motivated by the analogous situation in the
classical (linear) Calder\'on--Zygmund theory and several
preliminary existing results in the multilinear setting, which
we summarize in what follows.

As is well-known, the first to study the  commutator
$$
[b,\mathcal R^k](f):=\mathcal R^k (bf)-b\mathcal R^k(f)
$$
of the classical Riesz transforms $\mathcal R^k$ with
point-wise multiplication by a function $b$ were Coifman,
Rochberg and Weiss \cite{CRW}. They showed that $[b,\mathcal
R^k]$  is bounded on $L^p$ for some $p$ with $1< p<\infty$ if and only
if the symbol $b$ is in $\bmo(\mathbb{R}^n)$. Their result was then extended
to other non-degenerate Calder\'on--Zygmund operators by Janson
\cite{J} and Uchiyama \cite{U}. Moreover, Uchiyama showed that
$[b,\mathcal R^k]$  is compact on $L^p$ for some (then for all) $1< p<\infty$ if and
only if  the function $b$ is not just in $\bmo(\mathbb{R}^n)$ but actually in
$\cmo(\mathbb{R}^n)$.

In the multilinear setting, an interesting situation arises:
multilinear Calder\'on--Zygmund operators, their commutators,
and other related operators tend to be bounded also into $L^p$
spaces outside the Banach space situation. For example, in the
bilinear case a Calder\'on--Zygmund operator~$T$ in the sense of
Grafakos and Torres \cite{GT} (see also the references therein)
satisfies
 $$ T:L^{p_1} \times L^{p_2} \to  L^p,$$
for all $1<p_1<\infty$, $1<p_2<\infty$ and $1/p_1+1/p_2=1/p<2$.
This creates complications when studying the case of $p<1$ in
the target space, as some analytic tools (often depending on
duality) fail in this situation.  For this reason the case
$p>1$ and $p<1$ have been occasionally treated separately in
the literatures and by different arguments. For example, the
boundedness of the commutators
\begin{align*}
[b, T]_1(f,g):=\, & T(bf,g)-b T(f,g),\\
[b,T]_2(f,g):=\, & T(f,bg)-bT(f,g),
\end{align*}
of a bilinear Calder\'on--Zygmund
operator $T$ with a $\bmo$ function~$b$
was first obtained by P\'erez and Torres in \cite{PT} when
$p>1$, while the case of $p\leq 1$ was latter studied
independently by Tang \cite{tang} and Lerner et
al.~\cite{LOPTT}.  The compactness of the same commutators when
$b\in \cmo(\mathbb{R}^n)$ was obtained by B\'enyi and Torres in \cite{BT} but
only for  $p\geq 1$. Nonetheless, it was recently observed by
Torres and Xue \cite{TX} that the result also holds for
$1/2<p<1$.  The partial converse fact that the boundedness of
$[b, T]_1$ or $[b, T]_2$ for certain bilinear
Calder\'on--Zygmund operators forces $b$ to be in $\bmo(\mathbb{R}^n)$ was
first proved by Chaffee \cite{C} and was then also revisited by
Li and Wick \cite{LW} using different techniques. In both cases
the results are also under the assumption  $p>1$. Finally, in a very recent
manuscript posted in arXiv by Wang, Zhou and Teng \cite{WZT},
the result of Chaffee \cite{C} was extended to $1/2<p\leq 1$.

We will show in Theorem \ref{main} below that at least for the
bilinear Riesz transforms, the compactness of the commutators
forces the symbol $b$ to be in $\cmo(\mathbb{R}^n)$. Our work follows ideas
of Uchiyama \cite{U} and Chen, Ding and Wang~\cite{CDW} in the
linear case, as well as modification done in \cite{CT} for the
bilinear operators. We note however that the main difference
with respect to the work in \cite{CT}, and a difficulty we
overcome here, is that the operators in \cite{CT} are bilinear
fractional integral operators which are hence positively
defined, which is a property heavily used in \cite{CT} but
certainly completely failing for Calder\'on--Zygmund operators.
We refer the reader to \cite{CT} and the references therein for
more on commutators of fractional singular operators in both
linear and multilinear settings.

\bigskip

{\bf Acknowledgement}. Part of the work leading to this article
took place while the last two named authors were visiting the
Mathematical Sciences Research Institute (MSRI) at Berkeley in
February 2017, during the Harmonic Analysis program. This stay
at MSRI gave them a chance to combine different previous
efforts by all the colleagues involved. The authors would like
to thank the Institute and the organizers of the program for
providing the resources for such a fruitful opportunity to
carry out this research.

\section{Definitions}

As mentioned in the introduction, the space $\cmo(\R^n)$ is the
closure in the $\bmo(\R^n)$ topology of the space of infinitely
differentiable functions with compact support, denoted here by
$C^\infty_c(\R^n)$. For brevity, throughout the paper we denote
$L^p(\mathbb{R}^n)$ by $L^p$, and similarly for $\bmo$, $\cmo$
and $C_c^\infty$. Also, for convenience, we will use the $\bmo$
norm (modulo constants) defined for a locally integrable
function $b$ by
$$\|b\|_{\bmo}:=\sup_{Q}\dashint_Q |b(x)-b_Q|\, dx < \infty,$$
with the  supremum taken over all cubes $Q\in \R^n$  with edges
parallel to the coordinate axes, and where  for any locally
integrable function $f$ we use the standard notation $f_Q
=\dashint_Q f:=\frac{1}{|Q|}\int_Qf(x)\,dx$ for the average of
$f$ over $Q$. In addition, we recall (see \cite{U}) that $b \in
\bmo$ is in $\cmo$ if and only if
\begin{align}
  & \displaystyle\lim_{a\to0}\sup_{|Q|=a}\frac{1}{|Q|}\int_Q|b(x)-b_Q|\,dx=0, \label{1cmo}\\
   & \displaystyle\lim_{a\to\infty}\sup_{|Q|=a}\frac{1}{|Q|}\int_Q|b(x)-b_Q|\,dx=0, \quad\quad{\text{and}}\label{2cmo}\\
    & \displaystyle\lim_{|y|\to\infty}\frac{1}{|Q|}\int_Q|b(x+y)-b_Q|\,dx=0, \mbox{ for each } Q.\label{3cmo}
   \end{align}

For $x\in \R^n$ we will use the notation $x=(x^1,\dots ,x^n)$ and consider the $2n$ {\it bilinear Riesz transform operators} defined for $k=1,\dots,n$ by
\begin{align*}
\mathcal R^k_1(f,g)(x):=\, & \text{p.v.}\iint_{\R^{2n}}\frac{x^k-y^k}{(|x-y|^2+|x-z|^2)^{n+1/2}}\,f(y)g(z)\,dydz,\\
\mathcal R^k_2(f,g)(x):=\, & \text{p.v.}\iint_{\R^{2n}}\frac{x^k-z^k}{(|x-y|^2+|x-z|^2)^{n+1/2}}\,f(y)g(z)\,dydz.
\end{align*}
The name of these operators is justified by the fact that they can be ``obtained"  by considering the linear Riesz transforms in $\R^{2n}$ defined by
\begin{align*}
\mathcal R^k(F)(u):=\, & \text{p.v.}\int_{\R^{2n}}\frac{u^k-v^k}{|u-v|^{2n+1}}\,F(v)\,dv,
\end{align*}
where  $u=(u^1, \dots, u^{2n})$  and  $v=(v^1, \dots, v^{2n})$,  $k=1,\dots,2n$.
Note that setting $u=(x,x)$, $v=(y,z)$ with $x,y,z \in \R^n$,
and  $F(y,z)=f(y)g(z)$ leads, formally,  to the bilinear operators $\mathcal R^k_j$, $j=1,2$. For $k=1,\dots, n$,
$\mathcal R^k_1(f,g)(x)= \mathcal R^k(fg)(x,x)$, while $\mathcal R^k_2(f,g)(x)= \mathcal R^{k+n}(fg)(x,x)$.

The boundedness of the $\mathcal R^k_j$ operators from $L^{p_1}
\times L^{p_2}$ to $L^p$, for $1<p_1<\infty$, $1<p_2<\infty$
and $1/p_1+1/p_2=1/p<2$, is by now well-known. See for example
\cite{GT} and the references therein.

For $j=1,2$, and $k=1,\dots,n$,  the (first--order) commutators of the Riesz transform operators with a  symbol $b$  are given by
\begin{align}
\begin{split}\label{comm}
[b, \mathcal R^k_j]_1(f,g):=\, & \mathcal R^k_j(bf,g)-b \mathcal R^k_j(f,g),\\
[b,\mathcal R^k_j]_2(f,g):=\, & \mathcal R^k_j(f,bg)-b\mathcal R^k_j(f,g).
\end{split}
\end{align}
 Notice that $b \in \bmo$ is consistent with the fact that, by linearity, for any complex number $C$,
\begin{align*}
\begin{split}
[b-C, \mathcal R^k_j]_1(f,g)= [b, \mathcal R^k_j]_1(f,g),\\
[b-C,\mathcal R^k_j]_2(f,g)= [b, \mathcal R^k_j]_2(f,g),
\end{split}
\end{align*}
a fact that we will later use.

By the results mentioned in the introduction the boundedness of
any of these commutators from $L^{p_1} \times L^{p_2}$ to $
L^p$, for the full range of exponents $1<p_1<\infty$,
$1<p_2<\infty$ and $1/p_1+1/p_2=1/p<2$ is equivalent to $b$
being in $\bmo$. It is also known that they are compact for the
same range of exponents if in addition $b\in \cmo$. The new
result we shall present is the converse of this last statement.

\section{Characterization of compactness}

\begin{theorem}\label{main} Let $1<p_1<\infty$,
$1<p_2<\infty$ and
$\frac1{p}=\frac1{p_1}+\frac1{p_2}<2$.\footnote{We note that in
a first draft of this article we had stated Theorem \ref{main}
only for $p>1$. Although the computations in the proof (the
same presented here) work for all $1/2<p<\infty$, it was not
known at the time whether  the boundedness of the commutators
when $1/2<p\leq 1$  implies $b \in \bmo$, which is a condition
needed to jump start our arguments in the proof. Nothing else
in the proof depends on the value of $p>1/2$.  The recent result
in \cite{WZT} allows us now to state Theorem \ref{main} for the
full range of exponents without altering its proof.}
 Then each of the commutators in \eqref{comm} is a compact bilinear operator from
$L^{p_1} \times L^{p_2} \to  L^p$, if and only if $b\in \cmo$.
\end{theorem}

\begin{proof}
We only need to establish the necessity of $b\in \cmo$ since another direction was proved in \cite{BT} and \cite{TX} as noted in Introduction.
Moreover, by symmetry and a change of variables it is enough to
consider, for example,  $\mathcal R^1_1$ and $ [b, \mathcal
R_1^1]_1$. To simplify notation we denote $\mathcal
R^1_1$  by $\mathcal R$.

Fix exponents $p_1,p_2, p$ as in the statement of the theorem.
Since bilinear compact operators are bounded, if we assume
$\rb$ to be compact from $L^{p_1} \times L^{p_2} \to  L^p$ we
must have that $b \in \bmo$; see \cite{C} for $p>1$ and
\cite{WZT} for $1/2<p\leq 1$. So for convenience,  by
linearity, we may  assume that $b$ is real valued and with
$\|b\|_{\bmo}=1.$

We will follow very closely some arguments in \cite{U, CDW} and \cite{CT} to show that if $b$ fails to satisfy
one of the conditions \eqref{1cmo}--\eqref{3cmo}, then one
arrives at a contradiction with the compactness of the operator.
So $b$ must be in $\cmo$.  We notice, however,  that a main
difference in the arguments below, in particular with respect
to \cite{CDW} and \cite{CT}, is the fact  alluded to in the
introduction that the fractional integral operators considered
in those works are actually positive operators, while the
singular integrals studied here are not.  This requires a
modification in the lower estimate \eqref{Est2} proved below.

Assume that  $\{Q_j\}_j$ is a sequence of cubes such that
  \begin{equation}
  \frac{1}{|Q_j|}\int_{Q_j}|b(x)-b_{Q_j}|\,dx\geq\eps,\label{awayfromzero}
  \end{equation}
for some $\eps>0$ and all $j\in \mathbb{N}$. As in \cite{CDW}
and \cite{CT}, define two sequences of functions $\{f_j\}$ and
$\{g_j\}$ associated with the cubes $Q_j$ in the following way.
Let
    $$c_0:=|Q_j|^{-1}\int_{Q_j}\text{sgn}(b(y)-b_{Q_j})\,dy$$
and define
         $$f_j(y):=|Q_j|^{-\frac{1}{p_1}}\left({\rm sgn}(b(y)-b_{Q_j})-c_0\right)\chi_{Q_j}(y).$$
Here $\text{sgn}$ denotes the usual signum function. Define
also
$$g_j(y):=|Q_j|^{-\frac{1}{p_2}}\chi_{Q_j}(y).$$ These functions
satisfy the following properties
\begin{enumerate}[(a)]
  \item $\text{supp}\,f_j\subset Q_j$ and $\text{supp}\,g_j\subset Q_j,$
    \item $f_j(y)(b(y)-b_{Q_j})\geq0,$
    \item $\int f_j(y)\,dy=0,$
    \item$\int
        (b(y)-b_{Q_j})f_j(y)\,dy=|Q_j|^{-\frac{1}{p_1}}\int_{Q_j}|b(y)-b_{Q_j}|\,dy,$
    \item$|f_j(y)|\leq 2|Q_j|^{-\frac{1}{p_1}}$ and $|g_j(y)|\leq |Q_j|^{-\frac{1}{p_2}},$
\item$\|f_j\|_{L^{p_1}}\leq2$,
 \item$\|g_j\|_{L^{p_2}}=1.$
\end{enumerate}

  Let $\{y_j\}$ be the collection of centers of  the cubes  $\{Q_j\}$. Then for  all  $x\in\(2\sqrt n Q_j\)^{c}$ the following standard pointwise estimates hold:
    \begin{align}
    |\rb((b-b_{Q_j})f_j,g_j)(x)|&\lesssim |Q_j|^{\frac{1}{p'_1}+\frac{1}{p'_2}}|x-y_j|^{-2n},  \label{Est1}\\
    |\rb(f_j,g_j)(x)|&\lesssim |Q_j|^{\frac{1}{p'_1}+\frac{1}{p'_2}+\frac1n}|x-y_j|^{-2n-1},\label{Est3}
    \end{align}
where the constants involved are independent of $j, b, f_j,
g_j$ and $\eps$. Indeed, for all such $x$ and all $y\in Q_j$ we
have $|x-y| \approx |x-y_j| >0$, and hence by~(a) and (e),
\begin{align*}
    |\rb((b-b_{Q_j})&f_j,g_j)(x)|=\left|\ \iint_{\R^{2n}} \frac{ (x^1-y^1)(b(y)-b_{Q_j})f_j(y)g_j(z)}{\(|x-y|^2+|x-z|^2\)^{n+1/2}}\,dydz\right|\\
       &\lesssim \frac{1}{|Q_j|^{\frac{1}{p_1}+\frac{1}{p_2}}}|x-y_j|^{-2n}\int_{Q_j}\int_{Q_j}|b(y)-b_{Q_j}|\,dydz\\
    &\lesssim |Q_j|^{\frac{1}{p'_1}+\frac{1}{p'_2}}|x-y_j|^{-2n}\|b\|_{\bmo}\\
    &\lesssim |Q_j|^{\frac{1}{p'_1}+\frac{1}{p'_2}}|x-y_j|^{-2n}.
\end{align*}
On the other hand, using (a), (e), the cancellation property (c) of $f_j$ and the regularity of the kernel of the operator $\rb$,
\begin{align*}
    |&\rb(f_j,g_j)(x)|=\left|\iint_{\R^{2n}}\frac{ (x^1-y^1)f_j(y)g_j(z)}{\(|x-y|^2+|x-z|^2\)^{n+1/2}}\,dydz\right|\\
    &=\left| \int_{\R^n}\!\! \left( \int_{\R^n} \left( \frac{ (x^1-y^1)f_j(y)g_j(z)}{\(|x-y|^2+|x-z|^2\)^{n+1/2}} \right. \right. \right.\\
    &\lesssim  \hspace{4cm}  - \left. \left. \left.  \frac{(x^1-y_j^1)f_j(y)g_j(z)}{\(|x-y_j|^2+|x-z|^2\)^{n+1/2}} \right )\,dy \right)\, dz\right|\\
    &\lesssim \int_{Q_j}\int_{Q_j} \frac{|y-y_j||f_j(y)|g_j(z)}{\(|x-y_j|^2+|x-z|^2\)^{n+1}}\,dy dz\\
    &\lesssim \frac{|Q_j|^{\frac1n}}{|x-y_j|^{2n+1}}\int_{Q_j} \int_{Q_j} |f_j(y)|g_j(z)\,dydz\\
    &\lesssim |Q_j|^{\frac1{p'_1}+\frac{1}{p'_2}+\frac1n}|x-y_j|^{-2n-1}.
    \end{align*}

Next, we note that if $d_j$ is the side-length of $Q_j$ then
for all  positive numbers $\widetilde \gamma_1, \widetilde
\gamma_2$, with $\widetilde \gamma_2 = 8 \widetilde \gamma_1
\gg 1 $ there always exists a cube $\widetilde Q_j$ of side-length
$\frac{\widetilde \gamma_2}{4\sqrt n} d_j$    contained in  the
annulus
$$A=\{x\in \R^n: \widetilde \gamma_1  d_j < |x-y_j| < \widetilde \gamma_2 d_j\},$$
and such that $|x-y|\approx |x-y_j| \approx x^1- y_j^1 \approx
x^1-y^1>0$ for all $x\in \widetilde Q_j$ and all $y\in Q_j$. We
claim that for all such $x$,
\begin{equation}
    |\rb((b-b_{Q_j})f_j,g_j)(x)| \gtrsim \eps|Q_j|^{\frac{1}{p'_1}+\frac{1}{p'_2}}|x-y_j|^{-2n}\label{Est2},
\end{equation}
where again the constant involved is independent of $j, b, f_j,
g_j$ and $\eps$. To see~\eqref{Est2},  we use properties (b)
and (d) of $f_j$ to estimate
 \begin{align*}
    |\rb((b-b_{Q_j})&f_j,g_j)(x)|=\left|\iint_{\R^{2n}} \frac{(x^1-y^1)(b(y)-b_{Q_j})f_j(y)g_j(z)}{\(|x-y|^2+|x-z|^2\)^{n+1/2}}\,dydz\right|\\
    &\gtrsim |Q_j|^{1-\frac{1}{p_2}}|x-y_j|^{-2n} \int_{Q_j}(b(y)-b_{Q_j})f_j(y)\,dy\\
    &=  C_1 |Q_j|^{1-\frac{1}{p_2}}|x-y_j|^{-2n}|Q_j|^{1-\frac{1}{p_1}}\frac1{|Q_j|}\int_{Q_j}|(b(y)-b_{Q_j})|\,dy\\
    &\geq C_1 |Q_j|^{\frac{1}{p'_1}+\frac{1}{p'_2}}|x-y_j|^{-2n}\eps.
    \end{align*}

We continue to follow the computations in  \cite{U}, \cite{CDW}  and \cite{CT} and want to establish now that
there exist constants $\gamma_1,\gamma_2$ with $\gamma_2>\gamma_1>0$ and $\gamma_3>0$, depending only on $p_1,\ p_2,\ n$ and $\eps$, such that the following estimates hold:
    \begin{align}
    \(\int_{\gamma_1d_j<|x-y_j|<\gamma_2d_j}|[b,\rb]_1(f_j,g_j)(x)|^p\,dx\)^{\frac{1}{p}}&\geq\gamma_3\label{C11},\\
    \(\int_{|x-y_j|>\gamma_2d_j}|[b,\rb]_1(f_j,g_j)(x)|^p\,dx\)^{\frac{1}{p}}&\leq\frac{\gamma_3}4\label{C12}.
    \end{align}
In order to prove \eqref{C11} and \eqref{C12}, we first observe
that for every large enough number
$\widetilde\gamma_1>(\frac{1}{\ln \sqrt 2})^2$, by properties
(a) and (e) and the John--Nirenberg inequality,
    \begin{align*}
  &  \int_{|x-y_j|>\widetilde\gamma_1d_j}\left|(b(x)-b_{Q_j})\rb(f_j,g_j)(x)\right|^p\,dx\notag\\
    &\lesssim |Q_j|^{\(\frac1{p'_1}+\frac{1}{p'_2}+\frac1n\)p}
    \sum_{s=\lfloor\log_2(\widetilde\gamma_1)\rfloor}^\infty\int_{2^sd_j<|x-y_j|<2^{s+1}d_j}\frac{|b(x)-b_{Q_j}|^p}{|x-y_j|^{p(2n+1)}} \,dx\notag\\
    &\lesssim |Q_j|^{\(\frac1{p'_1}+\frac{1}{p'_2}+\frac1n\)p}\times \\
     &\,\,\,\,\,\,
    \sum_{s=\lfloor\log_2(\widetilde\gamma_1)\rfloor}^\infty2^{-s(2n+1)p}|Q_j|^{-\(2+\frac1n\)p}\int_{2^sd_j<|x-y_j|<2^{s+1}d_j}
    |b(x)-b_{Q_j}|^p\,dx\notag\\
    &\lesssim  |Q_j|^{\(\frac1{p'_1}+\frac{1}{p'_2} -2\)p}  \sum_{s=\lfloor\log_2(\widetilde\gamma_1)\rfloor}^\infty   2^{-s(2n+1)p}  s^p2^{sn}|Q_j|     \notag\\
    &\lesssim \sum_{s=\lfloor\log_2(\widetilde\gamma_1)\rfloor}^\infty 2^{-s\(2n-\frac np+\frac12\)p},
    \end{align*}
and hence by $1/p<2$,
  \begin{equation}
    \(\int_{|x-y_j|>\widetilde\gamma_1d_j}\left|(b(x)-b_{Q_j})\rb(f_j,g_j)(x)\right|^p\,dx\)^{\frac{1}{p}}
    \leq C_2 \widetilde\gamma_1^{-\(2n-\frac np+\frac12\)}.\label{Est4}
    \end{equation}
Next, for $\widetilde\gamma_2=8\widetilde\gamma_1$, using
\eqref{Est2} and \eqref{Est4}, we obtain the following estimates: for
$p\geq 1$,
    \begin{align}
 &\(\int_{\widetilde\gamma_1d_j<|x-y_j|<\widetilde\gamma_2d_j}|[b,\rb]_1(f_j,g_j)(x)|^p\,dx\)^{\frac1p} \nonumber\\
    &\quad\geq \(\int_{\widetilde\gamma_1d_j<|x-y_j|<\widetilde\gamma_2d_j}|\rb\((b-b_Q)f_j,g_j\)(x)|^p\,dx\)^{\frac1p} \nonumber\\
    &\quad\ \ \ \ -\(\int_{\widetilde\gamma_1d_j<|x-y_j|}|(b(x)-b_Q)\rb(f_j,g_j)(x)|^p\,dx\)^{\frac1p} \nonumber\\\
    &\quad\geq C_1 \eps|Q_j|^{\frac{1}{p_1'}+\frac{1}{p_2'}}\(\int_{\widetilde Q_j}|x-y_j|^{-2np}\,dx\)^{\frac1p} -C_2\widetilde\gamma_1^{-\(2n-\frac np+\frac12\)}\nonumber\\
    &\quad\geq C_1 \eps|Q_j|^{\frac{1}{p_1'}+\frac{1}{p_2'}} |\widetilde Q_j|^{\frac{1}{p}}  \widetilde \gamma_2^{-2n} |Q_j|^{-2} -C_2\widetilde\gamma_1^{-\(2n-\frac np+\frac12\)}\nonumber\\
    &\quad\geq C_1 \eps  (4 \sqrt n)^{-\frac{n}{p}} \widetilde\gamma_2^{-2n+\frac{n}{p}} -C_2  8^{\(2n-\frac np+\frac12\)}  \widetilde\gamma_2^{-\(2n-\frac np+\frac12\)}, \label{Est4'}
    \end{align}
    and for $1/2<p<1$,
 \begin{align}
    & \int_{\widetilde\gamma_1d_j<|x-y_j|<\widetilde\gamma_2d_j}|[b,\rb]_1(f_j,g_j)(x)|^p\,dx \nonumber\\
    &\quad\geq \int_{\widetilde\gamma_1d_j<|x-y_j|<\widetilde\gamma_2d_j}|\rb\((b-b_Q)f_j,g_j\)(x)|^p\,dx \nonumber\\
    &\quad\ \ \ \ -\int_{\widetilde\gamma_1d_j<|x-y_j|}|(b(x)-b_Q)\rb(f_j,g_j)(x)|^p\,dx \nonumber\\\
    &\quad\geq C_1 \eps^p|Q_j|^{\(\frac{1}{p_1'}+\frac{1}{p_2'}\)p}\int_{\widetilde Q_j}|x-y_j|^{-2np}\,dx -C_2\widetilde\gamma_1^{-\(2n-\frac np+\frac12\)p}\nonumber\\
    &\quad\geq C_1 \eps^p|Q_j|^{\(\frac{1}{p_1'}+\frac{1}{p_2'}\)p} |\widetilde Q_j|  \widetilde \gamma_2^{-2np} |Q_j|^{-2p} -C_2\widetilde\gamma_1^{-\(2n-\frac np+\frac12\)p}\nonumber\\
    &\quad\geq C_1 \eps^p  (4 \sqrt n)^{-n} \widetilde\gamma_2^{-2np+n} -C_2  8^{\(2n-\frac np+\frac12\)p}  \widetilde\gamma_2^{-\(2n-\frac np+\frac12\)p}. \label{Est4''}
    \end{align}
We can now use \eqref{Est4} and \eqref{Est4'} or \eqref{Est4''}
to replace  $\widetilde \gamma_1,\widetilde  \gamma_2$ with  $
\gamma_1$  sufficiently large and $\gamma_2=8 \gamma_1$, so
that \eqref{C11} and \eqref{C12} are verified for some
$\gamma_3>0$.

From here the arguments used in \cite{CT}, which in turn
followed the ones in~\cite{CDW}, can be repeated without any
changes. Namely, it is possible to construct sequences of cubes
$\{Q_j\}$ and functions $\{f_j\}$, $\{g_j\}$ in exactly the same
way as in \cite{CT} so that if any one of the
conditions~\eqref{1cmo}--\eqref{3cmo} were to be violated by
$b$, then we would arrive at a contradiction with the compactness
of $[b,\rb]_1$. The reader can easily follow the argument in
\cite[pp.491--493]{CT}, simply replacing $[b,I_\alpha]_1$
therein by $[b,\rb]_1$. To make our paper more self-contained, we now sketch an outline of the argument.

Using \eqref{Est1} and \eqref{Est3} it can be shown that  given
$\gamma_1$, $\gamma_2$, and $\gamma_3$ from  \eqref{C11}
and~\eqref{C12}, there exists a $\beta$ with
$0<\beta\ll\gamma_2$, depending on $p_1,\ p_2,\ n,$ and $\eps$,
such that for each measurable set
$$
E\subset\{x:\gamma_1d_j<|x-y_j|<\gamma_2d_j\}
$$
with $|E|/|Q_j|<\beta^n$, we get
\begin{align}
\left\|[b,\rb]_1(f_j,g_j)\right\|_{L^p(E)}\leq\frac{\gamma_3}{4}.\label{C2}
\end{align}

\noindent This estimate relies on the fact that the result of
Lemma~3.17 (1) of \cite{SW}, which is stated there for $p=1$,
also holds for all $p>0$, and hence also applies in our case,
where $p>1/2$. In \cite{CDW}, the estimate corresponding to our
\eqref{C2} was obtained using the case $p\geq 1$ of this lemma.

With this in hand, if we suppose that any one of the
conditions~\eqref{1cmo}--\eqref{3cmo} on $b$ fails, we can
construct a sequence of functions that will lead us to a
contradiction with the compactness of $[b,\rb]_1$. For
instance,  if $b$ does not satisfy \eqref{1cmo}, then there
exist some $\eps>0$ and a sequence $\{Q_j\}$ of cubes with
$|Q_j|\to0$ as $j\to\infty$ such that
        \begin{align*}
        \frac{1}{|Q_j|}\int_{Q_j}|b(y)-b_{Q_j}|\,dy\geq \eps,
        \end{align*}
         for every $j$.
First, select a subsequence, denoted by $\{Q_j^{(i)}\}$, so that the side-lengths satisfy
        \begin{align*}
        \frac{d_{j+1}^{(i)}}{d_{j}^{(i)}}&<\frac\beta{2\gamma_2}.
        \end{align*}
Next, let $f_j^{(i)}$ and $g_j^{(i)}$, as defined before, be the functions associated to the selected cubes~$Q_j^{(i)}$.
Finally, for each $k$, $m\in \mathbb{N}$, consider the sets:
\begin{align*}
  G&:=\{x:\gamma_1d^{(i)}_k<|x-y_k^{(i)}|<\gamma_2d_k^{(i)}\},\\
  G_1&:=G\setminus\{x:|x-y_{k+m}^{(i)}|\leq\gamma_2d_{k+m}^{(i)}\},\\
  G_2&:=\{x:|x-y_{k+m}^{(i)}|>\gamma_2d_{k+m}^{(i)}\}.
\end{align*}
The choice of the $Q_j^{(i)}$s implies that
\begin{align*}
\frac{|G_2^c\cap G|}{|Q_k^{(i)}|}\leq\beta^n; 
\end{align*}
see again \cite[p.307]{CDW}. For $p\geq 1$, we can then estimate
\begin{align}\label{eqn:noncompact}
  \|[b,\rb]_1&(f_k^{(i)},g_k^{(i)})-[b,\rb]_1(f_{k+m}^{(i)},g_{k+m}^{(i)})\|_{L^p}\nonumber\\
  &\geq\(\int_{G}\left|[b,\rb]_1(f_k^{(i)},g_k^{(i)})\right|^p - \int_{G_2^c\cap G}\left|[b,\rb]_1(f_k^{(i)},g_k^{(i)})\right|^p\)^{\frac1p}  \\
  &\,\,\,\,\,\,\,\,\,\,\,\,\,\,-\(\int_{G_2}\left|[b,\rb]_1(f_{k+m}^{(i)},g_{k+m}^{(i)})\right|^p\)^{\frac1p}.\nonumber
 \end{align}
Applying \eqref{C11}, \eqref{C2}, and \eqref{C12} respectively
to the three terms on the right-hand side
of~\eqref{eqn:noncompact}, we conclude
\begin{align*}
  \|[b,\rb]_1(f_k^{(i)},g_k^{(i)})-[b,\rb]_1(f_{k+m}^{(i)},g_{k+m}^{(i)})\|_{L^p}
&\geq\(\gamma_3^p-\frac{\gamma_{3}^p}{4^p}\)^{\frac1p}-\frac{\gamma_3}{4}\\
  &\geq \frac{\gamma_3}{2},
   \end{align*}
at least for $p\geq 1$.

In the case of $1/2 < p < 1$, a similar argument using the
reverse triangle inequality applied to the $p^{\text{th}}$
power of the left-hand side of~\eqref{eqn:noncompact} leads to
the lower bound
\begin{align*}
  \|[b,\rb]_1(f_k^{(i)},g_k^{(i)})-[b,\rb]_1(f_{k+m}^{(i)},g_{k+m}^{(i)})\|^p_{L^p}
\geq \left(1-\frac{2}{4^p}\right)\gamma_3^p.
   \end{align*}

This means that the image of the bounded set $\{(f_j,g_j)\}_j$
is not precompact, which contradicts our assumption on
$[b,\rb]_1$. The cases where $b$ does not satisfy
condition~\eqref{2cmo} or condition~\eqref{3cmo} are handled
similarly, and we conclude our proof here.
\end{proof}

\bigskip

\begin{remark} We observe that the arguments used for the Riesz
transforms~$\mathcal R^k_j$ in Theorem~\ref{main} also go
through in more generality. In order to get the lower bound (as
in formulas~\eqref{Est2} and \eqref{C11} above), one usually
uses the assumption that the kernel of the operator is
positive,  if not in the whole space, then at least in a
substantial portion of the space. For the Riesz
transforms~$\mathcal R^k_j$, although the kernel is not
positive, for each cube~$Q_j$ we can find another cube
$\widetilde{Q}_j$ such that $\widetilde{Q}_j$ lies in some
large annulus centered at the centre $y_j$ of~$Q_j$, and for
all $x\in \widetilde{Q}_j$ and $y$, $z\in Q_j$,
$$
    K(x-y, x-z)>0 \, \quad\quad{\rm and}\quad\quad \,|x-y| \approx |x-z| \approx |x-y_j|.
$$
This condition together with the Calder\'on--Zygmund conditions
on the size and regularity of the kernel suffice to obtain the
lower bound. This idea applies to certain other bounded
convolution-type singular operators, as we now discuss.

In the linear case, as is shown in Uchiyama's paper~\cite{U},
the Riesz transform can be replaced by convolution-type
singular integral operators with kernel of the form
$$
K(x)= \frac{ \Omega\left( x\right)}{|x|^n},
$$
where $\Omega$ is a homogeneous function of degree zero defined
on the unit sphere in~$\R^n$ and is sufficiently smooth. Such a
kernel is \emph{locally positive} in the sense that there is
some spherical cap~$A$ in the unit sphere~$S^{n-1}$ such that $
\Omega\left( x\right) > c_0 > 0$ for all $x \in A$.

Turning to the bilinear case, the arguments used for the
bilinear Riesz transforms~$\mathcal R^k_j$ in
Theorem~\ref{main} can be repeated for bounded convolution
bilinear operators with kernel of the form
$$
K(y,z)= \frac{ \Omega\left( \frac {(y,z)}{|(y,z)|}\right)}{(|y|^2+|z|^2)^n},
$$
where $\Omega$ is a homogeneous function of degree zero defined
on the unit sphere in~$\R^n\times\R^n$ and is sufficiently
smooth. We need more assumptions on this kernel than in the
linear case.

First, we assume that $1/K$ has an absolutely convergent
Fourier series in some ball in $\R^{2n}$. This assumption
guarantees that the boundedness of the commutator operator with
a function $b$ implies that $b \in \bmo$, by the main result
of~\cite{C}.

Second, we assume that there is some spherical cap $A$ on the
unit sphere $S^{n-1}$ such that $ \Omega\left( \frac
{(y,z)}{|(y,z)|}\right) > c_0 > 0$ for all $y$, $z \in A$. This
assumption enables us to get the lower bound
estimate~\eqref{Est2}. Indeed, given a cube $Q_j$ centred
at~$y_j$, we can find another cube $\widetilde{Q}_j$ such that
$\widetilde{Q}_j$ lies in some large annulus centered at $y_j$,
and for all $x\in \widetilde{Q}_j$ and all $y$, $z\in Q_j$,
$x-y$ and $x - z$ lie in an infinite cone in~$\R^n$ whose
vertex is at the origin and which passes through the cap~$A$.
From our assumption, it follows that
$$
K(x-y, x-z)>0 \, \quad\quad{\rm and}\quad\quad \,|x-y| \approx |x-z| \approx |x-y_j|
$$
for all $x \in  \widetilde Q_j$ and $y$, $z\in Q_j$. The
computations in the proof of Theorem \ref{main} can now be
repeated. We leave the details to the interested reader.
\end{remark}

\end{document}